\documentclass[reqno,11 pt]{amsart}
\usepackage{amsmath}
\usepackage[margin=1.0in]{geometry}
\usepackage{amscd,amsmath,amsxtra,amsthm,amssymb,stmaryrd,xr,mathrsfs,mathtools,enumerate,commath, comment, enumitem, graphicx}
\usepackage{amsfonts, amssymb, amsthm, amsmath, braket, hyperref, mathtools, comment, mathrsfs, enumitem, cleveref}
\usepackage{stmaryrd}
\usepackage{amsfonts}
\usepackage{orcidlink}
\usepackage{multirow}
\usepackage{xcolor}
\usepackage{commath}
\usepackage{comment}
\usepackage{tikz-cd}
\usepackage{tkz-graph}
\usepackage{colonequals}
\usepackage{hyperref}
\usepackage{graphicx}
\usepackage{bigstrut}
\usepackage{tabularx}
\usetikzlibrary{calc}
\usepackage{caption}
\numberwithin{equation}{section}
\usepackage{longtable, multirow, array}
\usepackage{cleveref}
\usetikzlibrary{calc,angles,quotes}
\crefname{section}{§}{§§}
\Crefname{section}{§}{§§}
\usetikzlibrary{decorations.pathreplacing,angles,quotes}
\usetikzlibrary{calc, decorations.pathreplacing}

\newcommand{\genlegendre}[4]{%
	\genfrac{(}{)}{}{#1}{#3}{#4}%
	\if\relax\detokenize{#2}\relax\else_{\!#2}\fi
}

\usepackage{longtable} % for 'longtable' environment
\usepackage{pdflscape} % for 'landscape' environment
\usepackage{booktabs}
\usepackage{hyperref}
\definecolor{vegasgold}{rgb}{0.77, 0.7, 0.35}
\definecolor{darkgoldenrod}{rgb}{0.72, 0.53, 0.04}
\definecolor{gold(metallic)}{rgb}{0.83, 0.69, 0.22}
\hypersetup{
 colorlinks=true,
 linkcolor=darkgoldenrod,
 filecolor=brown,      
 urlcolor=gold(metallic),
 citecolor=darkgoldenrod,
 }

\usepackage[all,cmtip]{xy}
\DeclareFontFamily{U}{wncy}{}
\DeclareFontShape{U}{wncy}{m}{n}{<->wncyr10}{}
\DeclareSymbolFont{mcy}{U}{wncy}{m}{n}
\DeclareMathSymbol{\Sh}{\mathord}{mcy}{"58}
\usepackage[T2A,T1]{fontenc}
\usepackage[OT2,T1]{fontenc}

\usepackage{tikz}
\usetikzlibrary{shapes.geometric}
\usetikzlibrary{decorations.markings}
\tikzset{every loop/.style={min distance=10mm,looseness=10}}
\tikzstyle{vertex}=[auto=left,circle,minimum size=1pt,inner sep=0pt]

\newtheorem{theorem}{Theorem}[section]
\newtheorem{lemma}[theorem]{Lemma}

\newtheorem*{theorem*}{Theorem}
\newtheorem*{ass*}{Assumption}
\newtheorem{definition}[theorem]{Definition}
\newtheorem{corollary}[theorem]{Corollary}
\newtheorem{remark}[theorem]{Remark}

\newtheorem{proposition}[theorem]{Proposition}

\newcommand{\Z}{\mathbb{Z}}

\newcommand{\p}{\mathfrak{p}}
\newcommand{\Q}{\mathbb{Q}}

\DeclareFontEncoding{OT2}{}{}
\DeclareSymbolFont{cyrletters}{OT2}{wncyr}{m}{n}
\DeclareMathSymbol{\Sha}{\mathalpha}{cyrletters}{"58}

\numberwithin{equation}{section}

\begin{document}

\title{Some Remarks on \texorpdfstring{$\tau$}{tau}-Congruent Numbers}
\author[Shamik Das]{Shamik Das}
\address[Das]{Department of Mathematics and Statistics, IIT Kanpur, India}
\email{shamikd@iitk.ac.in}

\author[Debajyoti De]{Debajyoti De}
\address[De]{Department of Mathematics, IIT Madras, India}
\email{debajyotide20@gmail.com}


\begin{abstract}
In this paper, we extend the work of \cite{Chahal} in several directions. We first determine all Heron triangles that tightly circumscribe the unit circle and the associated $\tau$-congruent numbers generated by them. We then characterize all rational right triangles that tightly circumscribe the unit ellipse and identify the corresponding congruent numbers. In addition, we study of the congruent numbers from the excircle opposite a vertex of a rational right triangle, that is, the circle tangent to one side of the triangle and to the extensions of the remaining two sides.

\end{abstract}

\maketitle

\section{Introduction}

Rational triangles, triangles whose side lengths are rational numbers have been studied since ancient times. Rational triangles with rational area are known as Heron triangles. The Indian mathematician Brahmagupta, 598--668 A.D., considered triangles with integral sides and integral area.  In this article, we focus primarily on $\tau$-congruent numbers, which are defined as follows:

\begin{definition}
Let $\tau=\tan\left(\frac{\theta}{2}\right)\in \mathbb{Q}_{>0},$
where $0<\theta<\pi$. A positive integer $n$ is called a
$\tau$-congruent number if there exists a triangle with rational side
lengths, having one angle equal to $\theta$, and area equal to $n$. Equivalently, there exist $a,b\in \mathbb{Q}_{>0}$ such that
\begin{equation*}
n=\frac{1}{2}ab\sin\theta
 =\frac{ab\tau}{1+\tau^{2}},  \quad \text{ where } \sin\theta=\frac{2\tau}{1+\tau^{2}}.    
\end{equation*}
 \end{definition}
Note that if $\tau=1$, then $n$ reduces to a classical congruent number; that is, $n$ occurs as the area of a right-angled Heron triangle. Given any positive integer $n$, it is well known that there always exists a triangle with rational
sides $(a,b,c)_\theta$ such that the area of the triangle is $n$, or equivalently, $n$ is a $\tau$-congruent number. In \cite{Goins-Maddox}, Goins and Maddox proved that the existence of rational points of order greater than $2$ on the elliptic curve
\[
E_{\tau,n}: y^2=x(x-n\tau)(x+n\tau^{-1})
\]
is equivalent to the existence of a Heron triangle of area $n$.  For a fixed integer $n$, the existence of infinitely many Heron triangles with area $n$ was investigated in \cite{Rusin}. Furthermore, \cite{Ghale-Das-Chakraborty} studied the possibility of a prime $p$ being a $\frac{1}{p}$-congruent number.

%Among these, the most classical and well studied are right triangles whose side lengths are rational numbers. Such triangles were already used in ancient Egypt, where ropes with knots were used to form precise right angles for building structures (see figure \ref{fig:egyptian_rope_protractor}). This naturally leads to a number theoretic question: which positive rational number arises as the area of a rational right triangle? If a right angle traingle has sides $a,b,c \in \Q$ with $a^2+b^2=c^2$ then the area is $n = \frac{1}{2}ab.$ This motivates the following definition: a positive rational number $n$ is called a congruent number if it is the area of a rational right triangle. Thus studying the congruent number is equivalent to studying the rational solution of $$a^2+b^2=c^2, \quad \frac{1}{2}ab=n.$$ Since scaling a triangle by a factor c multiplies its area by $c^2$, congruent numbers are defined only up to multiplication by rational squares. Hence it suffices to restrict our attention to positive square-free integers. \\

In a recent work, \cite{Chahal} the author studied congruent numbers generated by rational right triangles circumscribing the unit circle and analyzed the congruent numbers arising from such configurations. Motivated by this, we first consider Heron triangles with prescribed angle $\theta$ tightly circumscribing for unit circle. This gives rise to a Diophantine equation characterizing such triangles and from which we parametrize the sides of the triangles (see Theorem \ref{thm-X-theta}). This framework connects with the theory of $\tau$-congruent number, which provides a generalization of the classical congruent number from right triangles to arbitrary fixed angle rational triangles with rational $\tau=\tan(\frac{\theta}{2})$.

In another direction, we study rational right triangles that tightly circumscribe ellipses of area $\pi$, equivalently the unit ellipse (see Theorem \ref{thm-ellipse-triangle}). After a suitable normalization, the problem reduces to studying rational right triangles tightly enclosing the unit circle (see Section \ref{Rational Right Triangles Circumscribing Ellipses}). This enables us to investigate the congruent numbers arising from such triangles. Also, we study rational right triangles inscribed in the unit circle (See Section \ref{Circumcircle}). We conclude this article with a study of congruent numbers arising from the excircle opposite a vertex of a rational right triangle (see Section~\ref{Excircles}).

\section{Rational triangles with a prescribed angle and \texorpdfstring{$\theta$}{lg}-congruent numbers}

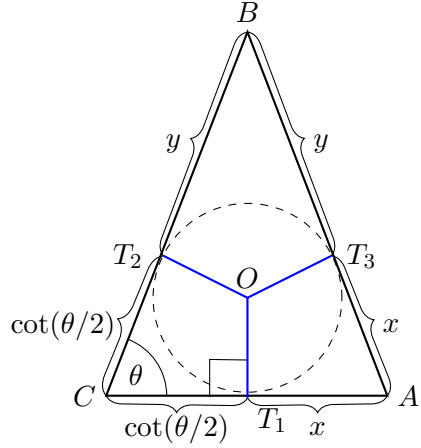
\begin{figure}[ht]
\centering

\begin{tikzpicture}[scale=1.85]

\coordinate (C) at (-1.01,0);
\coordinate (A) at (1,0);
\coordinate (B) at (0,2.6);
\coordinate (O) at (0,0.7);

\draw[thick] (C)--(A)--(B)--cycle;

\def\r{0.675}
\draw[dashed] (O) circle (\r);

\coordinate (T1) at (0,-0.015);
\coordinate (T2) at (-0.6,1.0);
\coordinate (T3) at (0.6,1.0);

\draw[blue,thick] (O)--(T1);
\draw[blue,thick] (O)--(T2);
\draw[blue,thick] (O)--(T3);

\draw pic[draw, angle radius=5mm]
{right angle=O--T1--C};

\draw pic["$\theta$", draw=black, angle radius=8mm]
{angle=A--C--B};

\node[left] at (C) {$C$};
\node[right] at (A) {$A$};
\node[above] at (B) {$B$};
\node[above] at (O) {$O$};

\node[below right] at (T1) {$T_1$};
\node[left=4pt] at (T2) {$T_2$};
\node[right=2pt] at (T3) {$T_3$};

\draw[decorate,
decoration={brace,amplitude=6pt,mirror}]
(C) -- (T1)
node[midway,yshift=-12pt]
{$\cot(\theta/2)$};

\draw[decorate,
decoration={brace,amplitude=6pt,mirror}]
(T1) -- (A)
node[midway,yshift=-12pt]
{$x$};

\draw[decorate,
decoration={brace,amplitude=7pt}]
(C) -- (T2)
node[midway,xshift=-27pt]
{$\cot(\theta/2)$};

\draw[decorate,
decoration={brace,amplitude=6pt}]
(T2) -- (B)
node[midway,xshift=-12pt]
{$y$};

\draw[decorate,
decoration={brace,amplitude=6pt}]
(B) -- (T3)
node[midway,xshift=12pt]
{$y$};

\draw[decorate,
decoration={brace,amplitude=6pt}]
(T3) -- (A)
node[midway,xshift=12pt]
{$x$};

\end{tikzpicture}

\caption{A triangle with incircle of radius $1$ and $\angle ACB=\theta$.}
\label{fig:theta_triangle_incircle}

\end{figure}

Let $\triangle ABC$ be triangle with incircle of radius $1$ and suppose the angle at vertex $C$ i.e. the angle $\angle ACB =\theta$, such that $\tau=\tan(\frac{\theta}{2}) \in \mathbb{Q}_{>0}$. Let $O$ be the incentre. Let $T_1$ be the point where the incircle touches the side $CA$. Similarly, $T_{2}$ and $T_{3}$ are defined. Want to prove that $CT_1= \cot(\theta/2)=\tau^{-1}$. Since $O$ is the incentre, it lies on the angle bisector of $\angle ACB$, hence $\angle T_1CO = \theta/2$. Also since $T_1$ is the tangency point, the radius to the tangent is perpendicular to the tangent line i.e. $OT_1 \perp CA$. In the right angle triangle $COT_1$, we have $\angle T_1CO= \dfrac{\theta}{2}$,\;\; $OT_1=1$. Then, 
$$\tan\left(\dfrac{\theta}{2}\right)= \dfrac{OT_1}{CT_1}= \dfrac{1}{CT_1} \implies CT_1= \cot\left(\dfrac{\theta}{2}\right)=\tau^{-1} .$$ 

\noindent Similarly, we obtain
$CT_2=\cot(\theta/2).$ Let $x$ denote the length of the tangent drawn from the vertex $A$ to the incircle along the sides $AC$ and $AB$. Then $AT_1=AT_3=x.$
Likewise, let $y$ denote the length of the tangent drawn from the vertex $B$ to the incircle along the sides $BC$ and $BA$. Then $BT_2=BT_3=y.$ Hence, the parametrization of the sides of the triangle is 
$BC= y + \tau^{-1}, \quad AC= x+\tau^{-1}, \quad  AB= x+y.$

We begin with the curve $X_{\tau}$ defined by the Diophantine equation:
\begin{equation}\label{Diophantine_equ}
    X_{\tau}: xy = \tau(x+y)+1.
\end{equation} 

By a rational (resp. integer) point on $X_{\tau}$ we mean a point whose coordinates are rational (resp. integer).  Hence solving \eqref{Diophantine_equ} for $y$ we obtain:
\begin{equation}\label{equation_of y}
    y=\dfrac{\tau x+1}{x-\tau}.
\end{equation}
Thus $X_{\tau}$ is a rational curve of genus zero. Since, for example, the rational point $(0,-1/\tau)$ lies on $X_{\tau}$, the curve has a rational point, and hence infinitely many rational points. Indeed, for every rational number $x\neq \tau$, the equation \eqref{equation_of y} produces a rational point on $X_{\tau}$. The following Lemma determine the integer points on $X_{\tau}$.

\begin{lemma}
\label{rational-points-on-X-tau}
    Assume $\tau^{-1} \in \Z_{>0}$.  Then among the infinitely many rational points on $X_{\tau}$, the integers point i.e. $(x,y) \in \Z^2$ are obtained from the divisor of $\tau^{-2}+1$. More precisely, the integer points are exactly 
    $$\left(\tau(d+1), \tau\left(\dfrac{\tau^{-2}+1}{d}+1\right)\right), \quad \text{ where } d \mid (\tau^{-2}+1) \;\;\; \text{ and } \;\;\; d \equiv -1 \pmod {\tau^{-1}}. $$ 
 In particular, when $\tau=1$, the only integer points are $(0,-1), (-1,0),(2,3), (3,2).$
\end{lemma}

\begin{proof}
    Multiply both side  of \eqref{Diophantine_equ}  by $\tau^{-2}$ and adding $1$ we have, $$(x\tau^{-1}-1)(y\tau^{-1}-1)=\tau^{-2}+1.$$ 
    Now if $(x,y)$ is an integer point then both $(x\tau^{-1}-1)$ and $(y\tau^{-1}-1)$ are divisor of $\tau^{-2}+1$. Let $x\tau^{-1}-1=d$ and $y\tau^{-1}-1 =\dfrac{\tau^{-2}+1}{d}$ for some $d \mid (\tau^{-2}+1).$ Hence $x= \tau(d+1)$, \;\; $y=\tau\left(\dfrac{\tau^{-2}+1}{d}+1\right)$. Hence $x, y \in \Z$ precisely when $d \equiv -1 \pmod{\tau^{-1}}$.
\end{proof}

\begin{theorem}
\label{thm-X-theta}
Let $X_{\tau}$ be the genus $0$ curve defined in \eqref{Diophantine_equ}, where $\tau \in \mathbb{Q}_{>0}$. Suppose that
\begin{equation}
\label{(a,b,c)=(x,y)}
 (a,\;b,\;c)=\left(y +\tau^{-1},\, x +\tau^{-1},\, x+y\right),   
\end{equation}
for a rational point $(x,y)$ on $X_{\tau}$. Then $(a,b,c)_{\theta}$ forms a Heron triangle having angle $\theta$ between the sides $a$ and $b$, and tightly circumscribing the unit circle, where $\tau=\tan\left(\frac{\theta}{2}\right).$
In particular, there exist infinitely many such triangles.
\end{theorem}

\begin{proof}
    %From \eqref{Side of triangle} it is clear that $(a,b,c)=(x+t, \,y+t, \,x+y) ~ \text{~ where ~} t=\cot\left(\dfrac{\theta}{2}\right).$
    Suppose that the triple $(a,b,c)$ as in \eqref{(a,b,c)=(x,y)}. By law of cosine it is enough to show that 
    $c^2=a^2+b^2-2ab\cos(\theta),$ where $\tau=\tan (\frac{\theta}{2})$. Note that, $\cos (\theta)= \frac{1-\tan^{2} (\frac{\theta}{2})}{1+\tan^{2} (\frac{\theta}{2})}=\frac{1-\tau^2}{1+\tau^2}.$
    Therefore,
    \begin{align*}
       a^2+b^2-2ab\cos(\theta)& = (y+ \tau^{-1})^2+(x+\tau^{-1})^2-2(y+ \tau^{-1})(x+\tau^{-1})\cdot \frac{1-\tau^2}{1+\tau^2}\\
       &=(y+ \tau^{-1})^2+(x+\tau^{-1})^2-\frac{2}{\tau^2} (\tau^2xy+\tau(x+y)+1) \cdot  \frac{1-\tau^2}{1+\tau^2}\\
       &= \frac{1}{\tau^2}\left[(\tau y+1)^2+(\tau x+ 1)^2-2xy(1-\tau^2)\right] \qquad \qquad \qquad \qquad (\text{by \eqref{Diophantine_equ}})\\
       &= \frac{1}{\tau^2}\left[\tau^2y^2+\tau^2x^2+2xy\tau^2+2(\tau(x+y)+1-xy)\right]\\
       &=(x+y)^2=c^2   \qquad \qquad\qquad\qquad\qquad \qquad\qquad\qquad\qquad\qquad (\text{by \eqref{Diophantine_equ}}).
    \end{align*}
 Finally, since $X_{\tau}$ is a genus $0$ rational curve with a rational point, it has infinitely many rational points. Hence \eqref{(a,b,c)=(x,y)} produces infinitely many rational triples $(a,b,c)$, and therefore infinitely many Heron triangles tightly circumscribing the unit circle with angle $\theta$.
\end{proof}

\begin{theorem}
Let $\theta\in(0,\pi)$ be such that  $\tau=\tan\left(\frac{\theta}{2}\right)\in\mathbb{Q}_{>0}.$ 
%$$\sin\theta=\frac{2\tau}{1+\tau^2}, \quad \cos\theta=\frac{1-\tau^2}{1+\tau^2}.$$
Suppose that $x,y\in\mathbb{Q}$ satisfy \eqref{Diophantine_equ}, and define
$a=y+\tau^{-1},\;\; b=x+\tau^{-1},\;\; c=x+y.$
Then $(a,b,c)_{\theta}$ is a Heron triangle with angle $\theta$ between the sides $a$ and $b$. Furthermore, the corresponding $\tau$-congruent number is given by $n=\tau xy.$
\end{theorem}

\begin{proof}
 Here, $\tau=\tan\left(\frac{\theta}{2}\right),$
and hence $\sin(\theta)=\frac{2\tau}{1+\tau^2}.$
 By \eqref{(a,b,c)=(x,y)}, the triple $(a,b,c)$ determines a Heron triangle with angle $\theta$ between the sides $a$ and $b$. Hence the associated $\tau$-congruent number is $n=\frac{ab\sin(\theta)}{2}.$ Therefore, we have
\begin{align*}
    n &=\frac{1}{2}(y+\tau^{-1})(x+\tau^{-1})\cdot \frac{2\tau}{1+\tau^2} \\
    &= \dfrac{(x\tau+1)}{\tau}\cdot\dfrac{(y\tau+1)}{\tau} \cdot \dfrac{\tau}{1+\tau^2}\\
    &= \dfrac{xy\tau^2+(x+y)\tau +1}{\tau(1+\tau^2)}\\
    &= \dfrac{xy\tau^2+xy-1+1}{\tau(1+\tau^2)} \qquad \qquad (\text{by \eqref{Diophantine_equ}})\\
    &= \dfrac{xy}{\tau} = \tau x y \pmod {\Q^{\times 2}}.
\end{align*}
Hence we are done.
% Using the Diophantine equation \eqref{Diophantine_equ}, this simplifies to
% $n=\tau xy.$
\end{proof}

\begin{remark}
    When $\theta=\frac{\pi}{2}$ (i.e. $\tau=1$), this construction reduces to the classical congruent number case studied in \cite{Chahal}, where the rational points on the curve $xy=x+y+1$ parametrize right triangles subscribing the unit circle and recover the usual congruent number correspondence. Thus our results generalize this framework to arbitrary angles $\theta$.
\end{remark}

\begin{table}[h]
\centering
\caption{Some $\tau$- congruent numbers}
\label{tab:theta congruent_numbers}
\renewcommand{\arraystretch}{1.1} 
\setlength{\tabcolsep}{18pt}      
\begin{tabular}{ccccc}
\hline
$\tau$ & $x>\tau$ & $y=\dfrac{\tau x+1}{x-\tau}$ & $n=\tau\cdot x \cdot y$ & $n \pmod{\mathbb{Q}^{\times 2}}$ \\
\hline
$1/2$ & $2$ & $4/3$ & $4/3$ & 3 \\
$1$ & $2$  & $3$  & $6$  & $6$ \\
$3/2$ & $2$ & $8$ & $24$ & $6$ \\
$2$ & $3$ & $7$ & $42$ & $42$ \\
%$1/3$ & $1$ & $2$ & $2/3$ & $6$ \\
$3$ & $4$ & $13$ & $156$ & $39$ \\
\hline
\end{tabular}
\end{table}

\section{Rational Right Triangles Circumscribing Ellipses of Area \texorpdfstring{$\pi$}{pi}}
\label{Rational Right Triangles Circumscribing Ellipses}

In this section, we generalize the problem of rational right triangles that enclose the unit circle to ellipses of area $\pi$, in short unit ellipse. Let 
\begin{equation}
\label{Ellipse-Ea}
 E_a:\frac{(x-a)^2}{a^2}+a^2\left(y-\frac{1}   {a}\right)^2=1,\qquad a\in \mathbb{Q}^{\times}.
\end{equation}

Then $E_a$ is an ellipse centered at $(a,\, \frac{1}{a})$, with semi axes $a$ and $\frac{1}{a}$ respectively. Hence, the area of $E_a$ is $\pi \cdot a \cdot \frac{1}{a} =\pi.$ Note that the coordinate axes $x=0$ and $y=0$ are tangent to $E_a$ at $\left(0,\frac{1}{a}\right)  \text{ and }  (a,0),$ respectively. Consequently, $E_a$ lies entirely in the first quadrant. Since the coordinate axes are tangent to $E_a$, it is natural to consider right triangles whose legs lie along these axes. Accordingly, let $\Delta$ be a rational right triangle with vertices $ (0,0),\, (u,0),\, (0,v),$
circumscribing $E_a$. Now consider the linear transformation
\begin{equation}
\label{Map-Ta}
   T_a:\mathbb{R}^2\to\mathbb{R}^2, \qquad T_a(x,y)=\left(\frac{x}{a},\,ay\right). 
\end{equation}
Geometrically, $T_a$ compresses the $x$-coordinate by a factor of $a$ and stretches the $y$-coordinate by the same factor. Since
$\det(T_a)=\frac1a\cdot a=1,$ the transformation preserves area.

\begin{figure}[ht]
\begin{center}
\resizebox{\textwidth}{!}{
\begin{tikzpicture}

\def\u{4.2}
\def\v{2.97}
\def\a{0.98}

\begin{scope}[scale=1.3]

\coordinate (O) at (0,0);
\coordinate (A) at (\u,0);
\coordinate (B) at (0,\v);

\draw[thick] (O)--(A)--(B)--cycle;
\draw (0,0) rectangle (0.15,0.15);

\draw[blue, thick]
    (0.99,1.03) ellipse ({\a} and {1/\a});

\node[left] at (B) {$(0,v)$};
\node[below] at (A) {$(u,0)$};
\node[below left] at (O) {$(0,0)$};
\node[blue] at (0.9,0.55) {$E_a$};x, y
\fill (0.99,1.03) circle (1.5pt);
\node[above right] at (0.99,1.03) {$(a,\tfrac{1}{a})$};
\fill (0.99,0) circle (1.5pt);
\node[below] at (0.99,0) {$(a,0)$};

\fill (0,1.03) circle (1.5pt);
\node[left] at (0,1/\a) {$(0,\tfrac1a)$};

\end{scope}

\draw[->, very thick] (4.2,2.0) -- (8.0,2.0);
\node at (6.1,2.6) {$T_a(x,y)=\left(\dfrac{x}{a},\,ay\right)$};

\begin{scope}[shift={(9.5,0)}, scale=1.3]

\coordinate (O2) at (0,0);
\coordinate (A2) at ({\u/\a},0);
\coordinate (B2) at (0,{\a*\v});

\draw[thick] (O2)--(A2)--(B2)--cycle;
\draw (0,0) rectangle (0.15,0.15);

\draw[red!75!black, thick]
    ({0.99/\a},{1.03*\a}) circle (1);

\node[left] at (B2) {$(0,av)$};
\node[below] at (A2) {$\left(\frac{u}{a},0\right)$};
\node[below left] at (O2) {$(0,0)$};
\node[red!75!black] at ({2.5/\a},{1.05*\a-0.45}) {$T_a(E_a)$};
\fill ({0.99/\a},{1.03*\a}) circle (1.5pt);
\node[above right] at ({0.99/\a},{1.03*\a}) {$(1,1)$};
\fill (1,0) circle (1.5pt);
\node[below] at (1,0) {$(1,0)$};

\fill (0,1) circle (1.5pt);
\node[left] at (0,1) {$(0,1)$};
\draw[dashed,gray] (1,1) -- (1,0);
\draw[dashed,gray] (1,1) -- (0,1);
\end{scope}

\end{tikzpicture}
}
\end{center}
\caption{The affine map $T_a(x,y)=\left(\frac{x}{a},\,ay\right)$ sends the ellipse $E_a$ centered at $(a,\tfrac1a)$ to the circle $(X-1)^2+(Y-1)^2=1$ centered at $(1,1)$.}
\label{Ellipse affine transformation}

\end{figure}
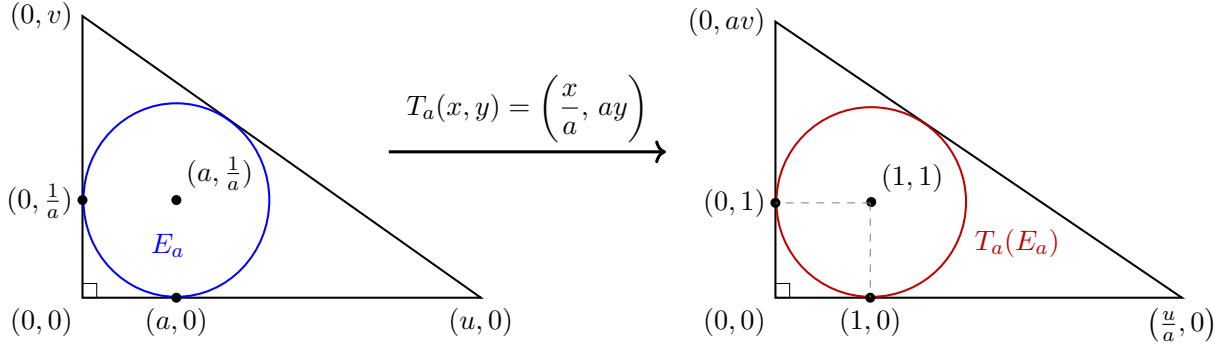

%Now let $(X,Y)=T_a(x,y)$. Then $X=\frac{x}{a},\quad Y=ay,$
%and hence $$(X-1)^2+(Y-1)^2=\left(\frac{x}{a}-1\right)^2+(ay-1)^2= \frac{(x-a)^2}{a^2}+a^2\left(y-\frac{1}{a}\right)^2.$$
Therefore,
$$(x,y)\in E_a \iff \frac{(x-a)^2}{a^2}+a^2\left(y-\frac{1}{a}\right)^2=1 \iff (X-1)^2+(Y-1)^2=1,$$
where $X=\frac{x}{a},\; Y=ay$. Thus, $T_a$ sends the ellipse $E_a$ bijectively onto the circle of radius $1$ centered at $(1,1)$.
Consequently, any rational right triangle  circumscribing $E_a$ is transformed into a right triangle  circumscribing the translated unit circle (see figure \ref{Ellipse affine transformation}). Then under the transformation $T_a$, the vertices of $\triangle$ are mapped to $
(0,0),\, \left(\frac{u}{a},0\right),\, (0,av).$ Thus, it is again a right triangle  circumscribing the translated unit circle. Moreover,
$\frac{1}{2}\cdot\frac{u}{a} \cdot av=\frac{1}{2}uv,$ so the area of the triangle is preserved. The following Lemma shows that it is indeed a rational right triangle.
%$$T_a(\triangle)=\left\{(0,0),\left(\frac{u}{a},0\right),(0,av)\right\}$$ is  \\

\begin{lemma}
\label{right-triangle-Ta}  
Under the transformation $T_a$ defined in \eqref{Map-Ta}, the rational right triangle with vertices $(0,0)$, $(u,0)$, and $(0,v)$ circumscribing the ellipse $E_a$ given in \eqref{Ellipse-Ea} is transformed into the rational right triangle with vertices
$(0,0), (\frac{u}{a},0), \text{ and } (0,av),$
circumscribing the unit circle centred at $(1,1)$.

\end{lemma}

\begin{proof}
In view of the above discussion, it suffices to show that the hypotenuse of the transformed right triangle is rational. Observe that, the hypotenuse of transformed triangle is the line joining $\left(\frac{u}{a},0\right)$ and $(0,av)$, whose equation is
\begin{equation}
\label{Hypotenuse equation}
\frac{x}{u/a}+\frac{y}{av}=1.
\end{equation}
By the definition of $T_a$ as in \eqref{Map-Ta}, the ellipse $E_a$ is transformed into the unit circle $(X-1)^2+(Y-1)^2=1$ centred at $(1,1)$. Since the original triangle $\Delta$ circumscribes $E_a$, the transformed triangle $T_a(\Delta)$ circumscribes this circle. Therefore the line \eqref{Hypotenuse equation} is tangent to the circle. Now the perpendicular distance from the centre $(1,1)$ to the tangent line \eqref{Hypotenuse equation} must equal the radius $1$. Writing the line \eqref{Hypotenuse equation} in the form $\frac{a}{u}x+ \frac{1}{av}y-1=0$, the distance formula gives $$\dfrac{\left|\frac{a}{u}+\frac{1}{av}-1\right|}
{\sqrt{\left(\frac{a}{u}\right)^2+\left(\frac{1}{av}\right)^2}} =1 \implies \sqrt{\left(\frac{u}{a}\right)^2+(av)^2}=\left|\frac{u}{a}+av-uv\right|
\in \mathbb{Q}.$$
 Hence the transformed triangle is a rational right triangle with vertices $(0,0), (\frac{u}{a},0), \text{ and } (0,av)$  circumscribing the unit circle centred at $(1,1)$.
\end{proof}

By the classification of rational right triangles circumscribing the unit circle \cite{Chahal}, there exists a rational point $(x,y)$ satisfying
$xy=x+y+1$ such that 
$\frac{u}{a}=x+1,\; av=y+1.$
Hence
$u=a(x+1),\;\; v=\frac{y+1}{a}.$ Conversely, given any rational point $(x,y)$ on the curve $xy=x+y+1$, the triangle with vertices $(0,0),\ (a(x+1),0),\ \left(0,\frac{y+1}{a}\right)$ is a rational right triangle circumscribing $E_a$. Thus rational right triangles circumscribing $E_a$ are parametrized by rational points on the same genus zero curve $xy=x+y+1$. Since $$ xy=x+y+1 \iff (x-1)(y-1)=2,$$
we parametrize all rational solutions by $x=1+t,~ y=1+\frac{2}{t},~  t\in\mathbb Q^\times.$ Substituting into the above formulas gives $
u=a(t+2), ~  v=\frac{1}{a}\left(2+\frac{2}{t}\right),$
which yields an explicit one parameter family of rational right triangles circumscribing $E_a$. Finally, since the transformation $T_a$ preserves area, the area of such a triangle is $$ \mathcal{A}=\frac{1}{2}uv =\frac{1}{2}\cdot a(t+2)\cdot \frac{2}{a}\left(1+\frac{1}{t}\right)=\frac{(t+2)(t+1)}{t}.$$
%which simplifies to $$\mathcal{A}=\frac{(t+2)(t+1)}{t}.$$
Thus, modulo $\mathbb{Q}^{\times 2}$, the congruent number generated by the triangle is
$\mathcal{A}\equiv t(t+1)(t+2).$ Therefore, the family of rational right triangles circumscribing the ellipse $E_{a}$
%$$E_a:\frac{(x-a)^2}{a^2}+a^2\left(y-\frac{1}{a}\right)^2=1$$
produces exactly the same congruent numbers (up to square factors) as the family of rational right triangles circumscribing the translated unit circle. Hence, we have the following theorems:

\begin{theorem}
\label{thm-ellipse-triangle}
Let $E_a$ be the ellipse defined in \eqref{Ellipse-Ea}, where $a\in \mathbb{Q}^{\times}$, so that $E_a$ has area $\pi$. A right triangle with rational vertices $(0,0)$, $(u,0)$, and $(0,v)$ circumscribes $E_a$ if and only if there exists a rational point $(x,y)$ on the curve
$xy=x+y+1,$ such that $u=a(x+1), \;\; v=\frac{y+1}{a}.$

\end{theorem}

\begin{corollary}
    The area of a rational right triangle circumscribing $E_a$ depends only on the associated rational point on $xy=x+y+1, $
    and is independent of $a$.
\end{corollary}

\begin{theorem} Suppose that $E_{a}$ is as in \eqref{Ellipse-Ea}.
\begin{itemize}
    \item[(a)] For every $a\in\mathbb Q^\times$, the set of rational right triangles circumscribing $E_a$ is in bijection with the rational points on the genus zero curve
   $xy=x+y+1.$ This bijection preserves area.
   \item[(b)] Up to rational equivalence, the congruent numbers obtained from triangles circumscribing $E_a$ are exactly the same as those arising from triangles circumscribing the unit circle.
   \item[(c)] A rational right triangle circumscribing $E_a$ has side lengths
$(u,\; v,\; \frac{u}{a}+av-uv),$
where $u$ and $v$ are the legs adjacent to the right angle $\pi/2$. In particular, when $a=1$, a rational right triangle circumscribing the unit circle has side lengths $(u,\; v,\; u+v-uv).$
\item[(d)] Product of three consecutive positive integers is a congruent number.
\end{itemize}
    
\end{theorem}

\begin{table}[h]
\centering
\caption{Congruent Numbers from Rational Right Triangles Circumscribing $E_a$}
\label{tab:congruent_numbers ellipse}
\renewcommand{\arraystretch}{1.4} 
\setlength{\tabcolsep}{20pt}      
\begin{tabular}{ccc}
\hline
$t$ & $\displaystyle \mathcal{A}=t(t+1)(t+2)$ & $\mathcal{A} \pmod{\mathbb{Q}^{\times 2}}$ \\
\hline
$1$ & $1 \cdot 2\cdot 3$  & $6$ \\

%$2$ & $(4\cdot 3)/2$ & $6$ \\[2pt]

$3$ & $3 \cdot 4\cdot 5$  & $15$ \\

$4$ & $4 \cdot 5\cdot 6$ & $30$ \\

$5$ & $5 \cdot 6\cdot 7$ & $210$ \\

$6$ & $6 \cdot 7\cdot 8$ & $42$ \\
\hline
\end{tabular}
\end{table}

\section{Circumcircle and Congruent Numbers}
\label{Circumcircle}

%\subsection{Circumcircle}

Let $\triangle ABC$ be a right angle triangle with $\angle C=\pi/2$ as in Figure \ref{fig:right_triangle_circumcircle}. Then the circumcentre of $\triangle ABC$ is the midpoint of the hypotenuse $AB$. Therefore, the circumradius is given by $R=\frac{AB}{2}$.   Let the right triangle have side lengths $(a,b,c)$ where $c$ is the hypotenuse. Then we have $c=2R$.

\begin{proposition}
Let $\triangle ABC$ be a rational right triangle with circumradius $R$. Then its side lengths can be parametrized as
$$
a=2R\frac{1-t^2}{1+t^2}, \qquad
b=2R\frac{2t}{1+t^2}, \qquad
c=2R,
$$
for some $t\in (0,1)\cap \mathbb{Q}$, where $c$ is the hypotenuse of the triangle.
\end{proposition}
\begin{proof}
Here we have 
$a^2+b^2=4R^2 \quad \text{ or, equivalently} \quad  \left( \frac{a}{2R} \right)^2 +\left( \frac{b}{2R} \right)^2=1.$ 
 So the point $\left( \frac{a}{2R},\, \frac{b}{2R} \right)$ lies on the unit circle. Using the rational parametrization $\left(\frac{1-t^2}{1+t^2},\, \frac{2t}{1+t^2}  \right)$ of the unit circle for parameter $t > 0$, we obtain 
the desired result (we can assume $0 < t < 1$ so that all the sides are positive).    
\end{proof}
\begin{corollary}
    With the above notation, the area of $\triangle ABC$ is $\mathcal{A}=\frac{ab}{2} =4R^2\frac{t(1-t^2)}{(1+t^2)^2}.$
    Consequently, $\mathcal{A} =t(1-t^2)$ is a congruent number.
\end{corollary}

\begin{figure}[h]
\centering

\begin{tikzpicture}[scale=0.7]

\coordinate (C) at (0,0);
\coordinate (A) at (4,0);
\coordinate (B) at (0,3);
\coordinate (O) at (2,1.5);

\draw[thick] (C)--(B)--(A)--cycle;
\draw[blue, thick] (O) circle (2.5);
\draw (0,0) rectangle (0.3,0.3);
\draw[dashed]  (A)--(B);

\fill (A) circle (2pt);
\fill (B) circle (2pt);
\fill (C) circle (2pt);
\fill (O) circle (2pt);

\node[below left] at (C) {$C$};
\node[below right] at (A) {$A$};
\node[above left] at (B) {$B$};
\node[above right] at (O) {$O$};

\node[below] at ($(A)!0.5!(B)$) {$c$};
\node[below] at ($(A)!0.5!(C)$) {$b$};
\node[above right] at ($(B)!0.55!(C)$) {$a$};

\end{tikzpicture}

\caption{A right triangle inscribed in its circumcircle.}
\label{fig:right_triangle_circumcircle}

\end{figure}
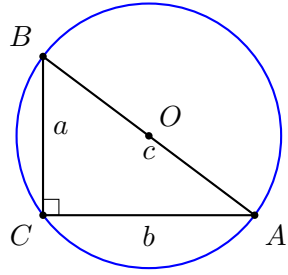

\section{Excircles and Congruent Numbers}
\label{Excircles}
 Let $\triangle ABC$ be a right triangle with with $\angle C=\pi/2$ and side lengths $BC=a$, $CA=b$, $AB=c$ where $c$ is the hypotenuse. Recall that the excircle opposite a vertex of a triangle is the circle tangent to one side of the triangle and to the extensions of the other two sides. These circles are referred to as the $a$-, $b$-, and $c$-excircles. For $\ell\in\{a,b,c\}$, let $r_\ell$ denote the corresponding exradius. We study rational right triangles satisfying $r_\ell=1.$

\begin{proposition}
With the above notation, let $s=\frac{a+b+c}{2}$ be the semiperimeter of $\triangle ABC$. Then the exradii are given by 
\begin{align*}
r_a &=s-b,\qquad r_b=s-a,\qquad
r_c=s, \\
\text{ equivalently, }\quad r_a&=\frac{a-b+c}{2},\qquad r_b=\frac{c-a+b}{2},\qquad
r_c=\frac{a+b+c}{2}.
\end{align*}
\end{proposition}

\begin{proof}

Since $\triangle ABC$ is right-angled at $C$, its area is $\Delta=\frac{ab}{2}.$ Using the formula $r_\ell=\frac{\Delta}{s-\ell},$ 
for the exradius opposite the side $\ell$ (see \cite{johnson2007advanced}), we obtain
$$r_a=\frac{ab}{-a+b+c} \qquad r_b=\frac{ab}{a-b+c},\qquad r_c=\frac{ab}{a+b-c}.$$
Now, using $c^2=a^2+b^2$, we have
$$(c+a-b)(c-a+b)=c^2-(a-b)^2=2ab ,$$
$$(a+b+c)(a+b-c)=(a+b)^2-c^2=2ab.$$
Hence
$$-a+b+c=\frac{2ab}{a-b+c}, \qquad a-b+c=\frac{2ab}{c-a+b},\qquad a+b-c=\frac{2ab}{a+b+c}.$$
Substituting into the above expressions for the exradii gives the desired result. 
\end{proof}

\begin{theorem}\label{thm:excircle-parametrization}
Let $\triangle ABC$ be a rational right triangle with side lengths $
BC=a,~ CA=b,~ AB=c,$ where $c$ is the hypotenuse.

\begin{enumerate}

\item[(a)] If the $a$-exradius satisfies $r_a=1$, then
$(a,\;b,\;c)=(x+1,\,z-1,\,z-x),$
where $(x,z)$ is a rational point on $X_a:\ z-zx-x=1.$

\item[(b)] If the $b$-exradius satisfies $r_b=1$, then
$(a,\;b,\;c)=(z-1,\,x+1,\,z-x),$
where $(x,z)$ is a rational point on $X_b:\ z(x-1)+x+1=0.$

\item[(c)] If the $c$-exradius satisfies $r_c=1$, then
$(a,\;b,\;c)=(1-y,\,1-x,\,x+y),$
where $(x,y)$ is a rational point on
$X_c:\ xy+x+y=1.$
\end{enumerate}
Each of the curves $X_a,X_b,$ and $X_c$ is a rational curve of genus zero. Consequently, there exist infinitely many rational right triangles with the corresponding exradius equal to $1$.
\end{theorem}

\begin{proof}
We prove the statement for the $c$-excircle, the proofs of the $a$- and $b$-excircle cases are analogous.

Let the $c$-excircle touch the hypotenuse $AB$ at $T$, and the extensions of $CA$ and $CB$ at $D$ and $E$, respectively (as shown in the figure \ref{fig:c-excircle}).
Set $AT=x,\; BT=y,\; CD=z.$ By equality of tangent lengths from an external point,
$AT=AD=x,\; BT=BE=y,\; CD=CE=z.$
Hence $c=AB=AT+TB=x+y,$
$a=BC=CE-BE=z-y,$ and  $b=CA=CD-AD=z-x.$
Therefore, $(a,\;b,\;c)=(z-y,\,z-x,\;x+y).$
Since $r_c=\frac{a+b+c}{2}=1,$ we obtain $(z-y)+(z-x)+(x+y)=2,$ and hence $z=1.$
Substituting this into the expressions for $a$ and $b$ gives $$(a,\;b,\;c)=(1-y,\;1-x,\;x+y).$$  Now imposing the Pythagorean condition $a^2+b^2=c^2$ gives $(1-y)^2+ (1-x)^2=(x+y)^2.$
After simplification we obtain the Diophantine curve, $xy+x+y=1.$   Hence rational points $(x,y)$ on the genus zero curve $xy+x+y=1$ with $0<x , y<1$ parameterize rational right triangles whose $c$-excircle has radius $1$.  Since $xy+x+y=1$, we have $x+y=1-xy$. Hence
$$ a+b=2-(x+y)=1+xy > 1-xy= x+y=c,$$
and similarly $a+c>b,\ b+c>a$, so $(a,b,c)$ forms a non-degenerate triangle. Moreover, $c-a=y(1-x)>0,\; c-b=x(1-y)>0,$ thus $c=x+y$ is the largest side.

The argument for the $a$-excircle and $b$-excircles are similar and yield the parametrizations: 
\begin{align*}
(a,\;b,\;c)&=(x+1,\;z-1,\;z-x), \quad X_a:\ z(1-x)-x-1=0,\\
\text{ and } \quad (a,\;b,\;c)&=(z-1,\;x+1,\;z-x),
\quad X_b:\ z(x-1)+x+1=0.    
\end{align*}
In each case, the defining curve admits a rational parametrization. Hence $X_a,X_b,$ and $X_c$ are rational curves of genus zero and therefore have infinitely many raional points. Therefore, each of the curves $X_a,X_b,$ and $X_c$ gives rise to infinitely many rational right triangles with the corresponding exradius equals to $1$. 
\end{proof}

Table~\ref{table:congruent-excircle-summary} summarizes the parametrizations of rational right triangles whose associated excircle has radius \(1\), together with the defining curve and the corresponding congruent number $\mathcal{A}$.

\begin{table}[h]
\centering
\renewcommand{\arraystretch}{1.3}
\setlength{\tabcolsep}{12pt}

\begin{tabular}{cccc}
\hline
Excircle & Parametrization & Defining curve & Area $\mathcal A$ \\
\hline
$a$-excircle &
$(x+1,\;z-1,\;z-x)$ &
$z(1-x)-x-1=0$ &
$\displaystyle x\frac{x+1}{1-x}$ \\
\hline
$b$-excircle &
$(z-1,\;x+1,\;z-x)$ &
$z(x-1)+x+1=0$ &
$\displaystyle x\frac{x+1}{1-x}$ \\
\hline
$c$-excircle &
$(1-y,\;1-x,\;x+y)$ &
$xy+x+y=1$ &
$\displaystyle x\frac{1-x}{1+x}$ \\
\hline
\end{tabular}

\caption{Summary of parametrizations for rational right triangles with unit excircle radius.}
\label{table:congruent-excircle-summary}
\end{table}

\begin{theorem}
Up to a rational square factor, every congruent number arises as the area of a rational right triangle whose $a$-, $b$-, or $c$-excircle has radius $1$.
\end{theorem}

\begin{proof}
    Let $\triangle$ be a rational right triangle with area $n$, and let $r_\ell$ denote the corresponding exradius, where $\ell\in\{a,b,c\}$. Scaling $\triangle$ by the factor $\frac1{r_\ell}$ produces a rational right triangle whose $\ell$-excircle has radius $1$. Since area scales quadratically, the resulting triangle has area $\frac{n}{r_\ell^2}.$ Hence $n \equiv \frac{n}{r_\ell^2} \pmod{\Q^{\times 2}},$ so the new area differs from $n$ by a rational square factor.
\end{proof}

\begin{corollary}
Every congruent number is represented, up to a rational square factor, by the area formulas appearing in Table~\ref{table:congruent-excircle-summary}.
\end{corollary}

\begin{figure}[ht]
\centering

\begin{tikzpicture}[scale=0.50]

\coordinate (C) at (0,0);
\coordinate (A) at (0,3);
\coordinate (B) at (4,0);

\draw[thick] (A)--(B)--(C)--cycle;
\draw[dashed] (A)--(0,7);
\draw[dashed] (B)--(7,0);

\node[left] at (A) {$A$};
\node[below] at (B) {$B$};
\node[below left] at (C) {$C$};

\node[left] at ($(A)!0.5!(C)$) {$b$};
\node[below] at ($(B)!0.5!(C)$) {$a$};
\node[above right] at ($(A)!0.5!(B)$) {$c$};

\draw (0.3,0) -- (0.3,0.3) -- (0,0.3);

\coordinate (Ic) at (6,6);
\draw[blue, thick] (Ic) circle (6);

\fill[blue] (Ic) circle (1.5pt)
node[above right] {$I_c$};

\coordinate (D) at (0,6);
\fill (D) circle (1.5pt);
\node[left] at (D) {$D$};

\coordinate (E) at (6,0);
\fill (E) circle (1.5pt);
\node[below] at (E) {$E$};

\coordinate (T) at ($(A)!0.6!(B)$);
\fill (T) circle (1.5pt);
\node[below left] at (T) {$T$};

\draw [decorate,
decoration={brace, amplitude=5pt, raise=2pt}]
(D) -- (A)
node [midway, left=6pt] {$x$};

\draw [decorate,
decoration={brace, amplitude=5pt, raise=2pt}]
(A) -- (T)
node [midway, above right=4pt] {$x$};

\draw [decorate,
decoration={brace, amplitude=5pt, raise=2pt}]
(B) -- (E)
node [midway, below=3pt] {$y$};

\draw [decorate,
decoration={brace, amplitude=5pt, raise=2pt}]
(T) -- (B)
node [midway, above right=2.5pt] {$y$};

\node[blue] at (8,8) {$c$-excircle};

\end{tikzpicture}

\caption{The $c$-excircle of a right triangle.}
\label{fig:c-excircle}

\end{figure}
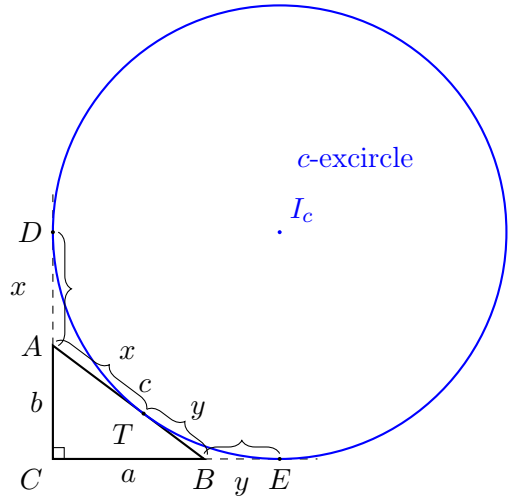

\subsection*{\it Acknowledgement} The first author acknowledges the support of the DST-INSPIRE Faculty Fellowship [DST/INSPIRE/04/2024/004189]. The second author acknowledges the Institute Fellowship and the excellent research environment provided by the Indian Institute of Technology Madras.

\bibliographystyle{plain}
\bibliography{references}

@article {Chahal,
    AUTHOR = {Chahal, Jasbir S.},
     TITLE = {Some remarks on rational right triangles},
   JOURNAL = {Expo. Math.},
  FJOURNAL = {Expositiones Mathematicae},
    VOLUME = {42},
      YEAR = {2024},
    NUMBER = {6},
     PAGES = {Paper No. 125623, 6},
      ISSN = {0723-0869,1878-0792},
   MRCLASS = {11D25 (51M04)},
  MRNUMBER = {4819800},
       DOI = {10.1016/j.exmath.2024.125623},
       URL = {https://doi.org/10.1016/j.exmath.2024.125623},
}

@article {Goins-Maddox,
    AUTHOR = {Goins, Edray Herber and Maddox, Davin},
     TITLE = {Heron triangles via elliptic curves},
   JOURNAL = {Rocky Mountain J. Math.},
  FJOURNAL = {The Rocky Mountain Journal of Mathematics},
    VOLUME = {36},
      YEAR = {2006},
    NUMBER = {5},
     PAGES = {1511--1526},
      ISSN = {0035-7596,1945-3795},
   MRCLASS = {14H52 (11G05 51M04)},
  MRNUMBER = {2285297},
MRREVIEWER = {Florian\ Luca},
       DOI = {10.1216/rmjm/1181069379},
       URL = {https://doi.org/10.1216/rmjm/1181069379},
}

@article {Rusin,
    AUTHOR = {Rusin, David J.},
     TITLE = {Rational triangles with equal area},
   JOURNAL = {New York J. Math.},
  FJOURNAL = {New York Journal of Mathematics},
    VOLUME = {4},
      YEAR = {1998},
     PAGES = {1--15},
      ISSN = {1076-9803},
   MRCLASS = {11G35 (11G05 14J20)},
  MRNUMBER = {1489407},
MRREVIEWER = {Joseph\ H.\ Silverman},
       URL = {http://nyjm.albany.edu:8000/j/1998/4_1.html},
}

@article {Ghale-Das-Chakraborty,
    AUTHOR = {Ghale, Vinodkumar and Das, Shamik and Chakraborty, Debopam},
     TITLE = {A {H}eron triangle and a {D}iophantine equation},
   JOURNAL = {Period. Math. Hungar.},
  FJOURNAL = {Periodica Mathematica Hungarica. Journal of the J\'anos Bolyai
              Mathematical Society},
    VOLUME = {86},
      YEAR = {2023},
    NUMBER = {2},
     PAGES = {530--537},
      ISSN = {0031-5303,1588-2829},
   MRCLASS = {11D25 (11G05 51M04)},
  MRNUMBER = {4591893},
MRREVIEWER = {Jasbir\ Singh\ Chahal},
       DOI = {10.1007/s10998-022-00491-5},
       URL = {https://doi.org/10.1007/s10998-022-00491-5},
}

@book{johnson2007advanced,
  author = {Roger A. Johnson},
  title = {Advanced Euclidean Geometry},
  publisher = {Dover Publications},
  address = {Mineola, New York},
  year = {2007},
  isbn = {9780486462370},
  note = {Reprint edition, p. 70}
}

\end{document}